\documentclass[preprint]{svjour3}

\let\vec\relax
\DeclareMathAccent{\vec}{\mathord}{letters}{"7E}

\usepackage{amsfonts}
\usepackage{amssymb}
\usepackage{mathtools}
\usepackage{xcolor}
\usepackage{lmodern}

\usepackage{graphicx}
\usepackage{epstopdf}

\usepackage[ruled,vlined,linesnumbered]{algorithm2e}

\usepackage{accents}

\journalname{}
\date{ \phantom{b} \vspace{45mm}\phantom{e}}
\def\makeheadbox{\hspace{-4mm}Version of 1 October 2020}

\newcommand*{\dt}[1]{ \accentset{\mbox{\large\bfseries .}}{#1}}
\DeclareMathOperator{\bigtimes}{{\hbox{\large\sf X}}}

\def\R{{\mathbb R}}

\def\eps{\varepsilon}

\def\P{\mathrm{P}}

\newcommand\bfr{{\mathbf r}}

\newcommand\bfI{{\mathbf I}}
\newcommand\bfA{{\mathbf A}}

\newcommand\bfD{{\mathbf D}}

\newcommand\bfF{{\mathbf F}}
\newcommand\bfG{{\mathbf G}}

\newcommand\bfK{{\mathbf K}}
\newcommand\bfL{{\mathbf L}}
\newcommand\bfM{{\mathbf M}}
\newcommand\bfN{{\mathbf N}}
\newcommand\bfR{{\mathbf R}}
\newcommand\bfS{{\mathbf S}}
\newcommand\bfU{{\mathbf U}}
\newcommand\bfV{{\mathbf V}}
\newcommand\bfY{{\mathbf Y}}
\newcommand\bfZ{{\mathbf Z}}

\def\eps{\varepsilon}

\def\phi{\varphi}

\author{Gianluca Ceruti, Christian Lubich}
\title{An unconventional robust integrator for \\dynamical low-rank approximation}
\date{}

\institute{G. Ceruti and Ch. Lubich \at
	Mathematisches Institut, Universit{\"a}t T{\"u}bingen, Auf der Morgenstelle 10, D-72076 T{\"u}bingen, Germany.
	\email{\{ceruti, lubich\}@na.uni-tuebingen.de}           
}

\begin{document}
	\maketitle
	
	\begin{abstract} We propose and analyse a numerical integrator that computes a low-rank approximation to large time-depen\-dent matrices that are either given explicitly via their increments or are the unknown solution to a matrix differential equation.  Furthermore, the integrator is extended to the approximation of  time-dependent tensors by Tucker tensors of fixed multilinear rank. The proposed low-rank integrator is different from the known projector-splitting integrator for dynamical low-rank approximation, but it retains the important robustness to small singular values that has so far been known only for the projector-splitting integrator. The new integrator also offers some potential advantages over the projector-splitting integrator: It avoids the backward time integration substep of the projector-splitting integrator, which is a potentially unstable substep for dissipative problems. It offers more parallelism, and it preserves symmetry or anti-symmetry of the matrix or tensor when the differential equation does.
	Numerical experiments illustrate the behaviour of the proposed integrator.
\end{abstract}

	\section{Introduction}
	For the approximation of  huge time-dependent matrices (or tensors) that are the solution to a matrix differential equation, dynamical low-rank approximation \cite{KochLubich07,KochLubich10} projects the right-hand side function of the differential equation to the tangent space of matrices (or tensors) of a fixed rank at the current approximation. This yields differential equations for the factors of an SVD-like decomposition of the time-dependent low-rank approximation. The direct numerical integration of these differential equations by standard methods such as explicit or implicit Runge--Kutta methods is highly problematic because in the typical presence of small singular values in the approximation, it leads to severe step size restrictions proportional to the smallest nonzero singular value. This difficulty does not arise with the projector-splitting integrator proposed in \cite{LubichOseledets}, which foregoes a direct time discretization of the differential equations for the factors and instead splits the orthogonal projection onto the tangent space, which is an alternating sum of subprojections. This approach leads to an efficiently implementable integrator that is robust to small singular values \cite{KieriLubichWalach,LubichOseledets}. It has been extended, together with its robustness properties, from the matrix case to Tucker tensors in \cite{Lubich:MCTDH,LubichVandWalach}, to tensor trains / matrix product states in \cite{LuOV15,HaLOVV16}, and to general tree tensor networks in~\cite{CeLW20}.
	
	In the present paper we propose and analyse a different integrator that is shown to have the same robust error behaviour as the projector-splitting integrator.
	This new integrator can apparently not be interpreted as a splitting integrator or be included in another familiar class of integrators. Its substeps look formally similar to those of the projector-splitting integrator but are arranged in a different, less sequential way. 
Like in the projector-splitting integrator, the differential equations in the substeps are linear if the original differential equation is linear, even though the projected differential equation becomes nonlinear.
The new integrator bears some similarity also to the constant-mean-field integrator of \cite{BeckMeyer} and the splitting integrator of \cite{KhoromskijOseledetsSchneider}. 
	
	Beyond the robustness to small singular values, the new integrator has some favourable further properties that are not shared with the projector-splitting integrator. Maybe most importantly, it has no backward time integration substep as in the projector-splitting integrator. This appears advantageous in strongly dissipative problems, where the backward time integration step represents an unstable substep. Moreover, the new integrator has enhanced parallelism in its substeps, and in the Tucker tensor case even a reduced serial computational cost. It preserves symmetry or anti-symmetry of the matrix or tensor when the differential equation does. It reduces to the (anti-)symmetry-preserving low-rank integrator of \cite{CeL20} in this case.
	
	On the other hand, unlike the projector-splitting integrator it cannot be efficiently extended to a time-reversible integrator. When applied to the time-dependent Schr\"odinger equation (as an integrator for the MCTDH method of quantum molecular dynamics; cf.~\cite{Beck-etal:MCTDH,BeckMeyer,Lubich:MCTDH}),
the new integrator preserves the norm, but it has no energy conservation as shown in \cite{Lubich:MCTDH} for the projector-splitting integrator. 	
	
	In Section 2 we recapitulate dynamical low-rank approximation and the projector-splitting integrator for the matrix case. We restate its exactness property and its robust error bound.
	
	In Section 3 we present the new low-rank matrix integrator and show that it has the same exactness property and robust error bound as the matrix projector-splitting integrator.
	
	In Section 4 we recapitulate dynamical low-rank approximation by Tucker tensors of fixed multilinear rank and the extension of the projector-splitting integrator to the Tucker tensor case.
	
	In Section 5 we present the new low-rank Tucker tensor integrator and show that it has the same exactness property and robust error bound as the Tucker tensor projector-splitting integrator.
	
	In Section 6 we illustrate the behaviour of the new low-rank matrix and Tucker tensor integrators by numerical experiments.
	
	While we describe the integrator for real matrices and tensors, the algorithm and its properties extend in a straightforward way to complex matrices and tensors,   
requiring only some care in using transposes $\bfU^\top$ versus adjoints $\bfU^*=\overline \bfU^\top$.
	
Throughout the paper, we use the convention to denote matrices by boldface capital letters and tensors by italic capital letters.

	\section{Recap: the matrix projector-splitting integrator}
	Dynamical low-rank approximation of time-dependent matrices \cite{KochLubich07} replaces the exact solution $\bfA(t)\in\R^{m\times n}$ of a (too large) matrix differential equation
		\begin{equation} \label{eq:fullEq-mat}
	\dt{\bfA}(t) = \bfF(t, \bfA(t)), 
	\qquad
	\bfA(t_0) = \bfA_0 
	\end{equation}
	by the solution $\bfY(t)\in\R^{m\times n}$ of rank $r$ of the differential equation projected to the tangent space of the manifold of rank-$r$ matrices at the current approximation,
	\begin{equation} \label{eq:projEq}
	\dt{\bfY}(t) = \P(\bfY(t))\bfF(t, \bfY(t)),
	\qquad
	\bfY(t_0) = \bfY_0,
	\end{equation}
	where the initial rank-$r$ matrix $\bfY_0$ is typically obtained from a truncated singular value decomposition (SVD) of $\bfA_0 $.
	(We note that $\bfF(t,\bfY) = \dt\bfA(t)$ if $\bfA(t)$ is given explicitly.) For the actual computation with rank-$r$ matrices, they are represented in a non-unique factorized SVD-like form
	\begin{equation} \label{USV}
	\bfY(t) = \bfU(t)\bfS (t)\bfV(t)^\top, 
	\end{equation}
	where the slim matrices $\bfU(t)\in \R^{m\times r}$ and $\bfV(t)\in \R^{n\times r}$ each have $r$ orthonormal columns, 
	and the small matrix $\bfS(t)\in \R^{r\times r}$ is invertible.

	The orthogonal tangent space projection $\P(\bfY)$ can be written explicitly as  an alternating sum of three subprojections onto the co-range, the intersection of co-range and range, and the range of the rank-$r$ matrix $\bfY$ \cite{KochLubich07}. The projector-splitting integrator of \cite{LubichOseledets} splits the right-hand side of \eqref{eq:projEq} according to the three subprojections in the stated ordering and solves the subproblems consecutively in the usual way of a Lie--Trotter or Strang splitting. This approach yields an efficient time-stepping algorithm that updates the factors in the SVD-like decomposition of the rank-$r$ matrices in every time step, alternating between solving differential equations for matrices of the dimension of the factor matrices and orthogonal decompositions of slim matrices.

	One time step from  $t_0$ to $t_1=t_0+h$  starting from a factored rank-$r$ matrix 
	$\bfY_0=\bfU_0\bfS_0\bfV_0^\top$ proceeds as follows:
	
	\begin{enumerate}
		\item 
		\textbf{K-step} : Update $ \bfU_0 \rightarrow \bfU_1, \bfS_0 \rightarrow \hat{\bfS}_1$ \\
		Integrate from $t=t_0$ to $t_1$ the $m \times r$ matrix differential equation
		$$ \dot{\textbf{K}}(t) = \bfF(t, \textbf{K}(t) \bfV_0^\top) \bfV_0, \qquad \textbf{K}(t_0) = \bfU_0 \bfS_0.$$
		Perform a QR factorization $\textbf{K}(t_1) = \bfU_1 \hat{\bfS}_1$. 
		
		\item
		\textbf{S-step} : Update $ \hat{\bfS}_1 \rightarrow \tilde{\bfS}_0$ \\
		Integrate from $t=t_0$ to $t_1$ the $r \times r$ matrix differential equation
		$$ \dot{\bfS}(t) = - \bfU_1^\top \bfF(t, \bfU_1 \bfS(t) \bfV_0^\top) \bfV_0, \qquad \bfS(t_0) = \hat{\bfS}_1,$$
		and set $\tilde{\bfS}_0 =\bfS(t_1)$.
		
		\item
		\textbf{L-step} : Update $ \bfV_0 \rightarrow \bfV_1, \tilde{\bfS}_0 \rightarrow \bfS_1$ \\
		Integrate from $t=t_0$ to $t_1$ the $n \times r$ matrix differential equation
		$$  \dot{\textbf{L}}(t) =\bfF(t, \bfU_1 \textbf{L}(t)^\top)^\top  \bfU_1, \qquad \textbf{L}(t_0) = \bfV_0 \tilde{\bfS}_0^\top. $$
		Perform a QR factorization $\textbf{L}(t_1) = \bfV_1 \bfS_1^\top$.
	\end{enumerate} 
	
	Then, the approximation after one time step is given by 
	$$ \bfY_1 = \bfU_1 \bfS_1 \bfV_1^\top .$$
	To proceed further,  $\bfY_1$ is taken as the starting value for the next step, and so on.

	The projector-splitting integrator has very favourable properties. First, it reproduces rank-$r$ matrices exactly.

	\begin{theorem}[{Exactness property, \cite[Theorem 4.1]{LubichOseledets}}]
		\label{thm:proj-split-exact}
		Let $\bfA(t) \in \mathbb{R}^{m \times n}$ be of rank~$r$  for $t_0 \leq t \leq t_1$,
		so that $\bfA(t)$ has a factorization \eqref{USV}, $\bfA(t)=\bfU(t)\bfS(t)\bfV(t)^\top$. 
		Moreover, assume  that the $r\times r$ matrices  $\bfU(t_1)^\top \bfU(t_0)$ and $\bfV(t_1)^\top \bfV(t_0)$ are invertible. With $\bfY_0 = \bfA(t_0)$, the projector-splitting integrator for 
		$\dt\bfY(t)=\P(\bfY(t))\dt \bfA(t)$ is then exact: $ \bfY_1 = \bfA(t_1)$.
	\end{theorem}
	
Even more remarkable, the algorithm is robust to the presence of small singular values of the solution or its approximation, as opposed to standard integrators applied to \eqref{eq:projEq} or the equivalent differential equations for the factors $\bfU(t)$, $\bfS(t)$, $\bfV(t)$, which contain a factor $\bfS(t)^{-1}$ on the right-hand sides \cite[Prop.\,2.1]{KochLubich07}. 
	The appearance of small singular values is ubiquitous in low-rank approximation, because the smallest singular value retained in the approximation cannot be expected to be much larger than the largest discarded singular value of the solution, which is required to be small for good accuracy of the low-rank approximation.
	
	\begin{theorem}[{Robust error bound, \cite[Theorem 2.1]{KieriLubichWalach}}]
		\label{thm:proj-split-robust}
		Let $\bfA(t)$ denote the solution of the matrix differential equation \eqref{eq:fullEq-mat}. Assume that  the following conditions hold in the Frobenius norm $\|\cdot\|=\|\cdot\|_F$:
		\begin{enumerate}
			\item 
			$\bfF$ is Lipschitz-continuous and bounded: for all $\bfY, \widetilde{\bfY} \in \mathbb{R}^{m \times n}$ and $0\le t \le T$,
			$$ 
			\| \bfF(t, \bfY) - \bfF(t, \widetilde{\bfY}) \| 
			\leq
			L \| \bfY - \widetilde{\bfY} \|,
			\qquad
			\| \bfF(t, \bfY) \| \leq B \ .
			$$
			
			\item
			The non-tangential part of $\bfF(t, \bfY)$ is $\varepsilon$-small:
			$$
			\| (\bfI - \P(\bfY)) \bfF(t, \bfY) \| \le \eps
			$$
			for all $\bfY \in \mathcal{M}$ in a neighbourhood of $\bfA(t)$ and $0\le t \le T$.
			
			\item
			The error in the initial value is $\delta$-small:
			$$
			\| \bfY_0 - \bfA_0 \| \le \delta.
			$$
		\end{enumerate}	
		Let $\bfY_n$ denote the rank-$r$ approximation to $\bfA(t_n)$ at $t_n=nh$ obtained after n steps of the projector-splitting integrator with step-size $h>0$.
		Then, the error satisfies for all $n$ with $t_n =  nh \leq T$
		$$ \| \bfY_n - \bfA(t_n) \| \leq c_0\delta + c_1 \varepsilon + c_2 h ,$$	
		where the constants $c_i$ only depend on $L, B,$ and $T$. In particular, the constants are independent of singular values of the exact or approximate solution. 
	\end{theorem}
	
	 In \cite[Section 2.6.3]{KieriLubichWalach} it is shown that an inexact solution of the matrix differential equations in the projector-splitting integrator
	leads to an additional error that is bounded in terms of the local errors in the inexact substeps, again with constants that do not depend on small singular values.
	
	Numerical experiments with the matrix projector-splitting integrator and comparisons with standard numerical integrators are reported in \cite{LubichOseledets,KieriLubichWalach}. 
	
	\section{A new robust low-rank matrix integrator}
	\label{sec:Par}
		We now present a different integrator that has the same exactness and robustness properties as the projector-splitting integrator but which differs in the following favourable properties:
	\begin{enumerate}
	\item The solution of the differential equations for the $m\times r$ and $n\times r$ matrices can be done in parallel, and also the two QR decompositions can be done in parallel.
	\item The differential equation for the small $r\times r$ matrix is solved forward in time, not backwards.
	\item The integrator preserves (skew-)symmetry if the differential equation does.
	\end{enumerate}
While item 1.~can clearly speed up the computation, item 2.~is of interest for strongly dissipative problems, for which the $S$-step in the projector-splitting algorithm with the minus sign in the differential equations is an unstable substep of the algorithm. This  does not appear in the new algorithm. We mention that in \cite{Bachmayer-etal:parabolic}, the problem of the backward substep for parabolic problems
has recently been addressed in a different way.

On the other hand, contrary to the projector-splitting integrator, there is apparently no efficient way to construct a time-reversible integrator from this new integrator.

\subsection{Formulation of the algorithm} 
\label{subsec:alg-newint}
	One time step of integration from time $t_0$ to $t_1=t_0+h$  starting from a factored rank-$r$ matrix 
	$\bfY_0=\bfU_0\bfS_0\bfV_0^\top$ computes an updated rank-$r$ factorization $\bfY_1=\bfU_1\bfS_1\bfV_1^\top$ as follows.
	
	\begin{enumerate}
		\item 
		Update $ \bfU_0 \rightarrow \bfU_1$ and $ \bfV_0 \rightarrow \bfV_1$ in parallel:
		\\[2mm]
		\textbf{K-step}:
		Integrate from $t=t_0$ to $t_1$ the $m \times r$ matrix differential equation
		$$ \dot{\textbf{K}}(t) = \bfF(t, \textbf{K}(t) \bfV_0^\top) \bfV_0, \qquad \textbf{K}(t_0) = \bfU_0 \bfS_0.$$
		Perform a QR factorization $\textbf{K}(t_1) = \bfU_1 {\bfR}_1$ and compute the $r\times r$ matrix $\bfM= \bfU_1^\top \bfU_0$.
				\\[2mm]
		\textbf{L-step} : 
		Integrate from $t=t_0$ to $t_1$ the $n \times r$ matrix differential equation
		$$  \dot{\textbf{L}}(t) =\bfF(t, \bfU_0 \textbf{L}(t)^\top)^\top  \bfU_0, \qquad \textbf{L}(t_0) = \bfV_0 {\bfS}_0^\top. $$
		Perform a QR factorization $\textbf{L}(t_1) = \bfV_1 \widetilde{\bfR}_1$ and compute the $r\times r$ matrix $\bfN= \bfV_1^\top \bfV_0$.
		
		\item
		Update ${\bfS}_0 \rightarrow {\bfS}_1$\,: \\[2mm]
		\textbf{S-step} :  Integrate from $t=t_0$ to $t_1$ the $r \times r$ matrix differential equation
		$$ \dot{\bfS}(t) =  \bfU_1^\top \bfF(t, \bfU_1 \bfS(t) \bfV_1^\top) \bfV_1, 
		\qquad \bfS(t_0) = \bfM {\bfS}_0 \bfN^\top,
		$$
		and set ${\bfS}_1 =\bfS(t_1)$.

	\end{enumerate} 
	
The $m\times r$, $n\times r$ and $r\times r$ matrix differential equations in the substeps are solved approximately using a standard integrator, e.g., an explicit or implicit Runge--Kutta method or an exponential integrator when $\bfF$ is dominantly linear.

We note that the L-step equals the K-step for the transposed function $\bfG(t, \bfY)=\bfF(t,\bfY^\top)^\top$ and transposed starting values.
Unlike the projector-splitting algorithm, the triangular factors of the QR-decompositions are not reused. The S-step can be viewed as a Galerkin method for the differential equation \eqref{eq:fullEq-mat} in the space of matrices $\bfU_1 \bfS \bfV_1^\top$ generated by the updated basis matrices. In contrast to the projector-splitting integrator, there is no minus sign on the right-hand side of the differential equation for~$\bfS(t)$. We further note that $\bfU_1$ of the new integrator is identical to $\bfU_1$ of the projector-splitting integrator, but $\bfV_1$ is in general different.
%

\begin{remark} There exists a modification where all three differential equations for $\bfK$, $\bfL$ and $\bfS$ can be solved in parallel. That variant solves the K- and L-steps as above, but in the S-step it solves instead the $r \times r$ matrix differential equation 
$$ \dot{\bfS}(t) =  \bfU_0^\top \bfF(t, \bfU_0 \bfS(t) \bfV_0^\top) \bfV_0, 
		\qquad \bfS(t_0) = {\bfS}_0
$$
and finally sets 
$$
{\bfS}_1 =  \bfM^{-\top}\, \bfS(t_1)\, \bfN^{-1}.
$$ 
This modified integrator can be shown to have the same exactness property as proved below for the integrator formulated above, and also  a similar robust error bound  {\it under the condition that} the inverses of the matrices $\bfM$ and $\bfN$ are bounded by a constant. This condition can, however, be guaranteed only for step sizes that are small in comparison to the smallest nonzero singular value. In our numerical experiments this method did not behave as reliably as the method proposed above, and despite its interesting properties it will therefore not be further discussed in the following.
\end{remark}

\subsection{Exactness property and robust error bound}

We will prove the following remarkable results for the integrator of Section~\ref{subsec:alg-newint}.

\begin{theorem}\label{thm:new-exact}
The exactness property of Theorem~$\ref{thm:proj-split-exact}$ holds verbatim also for the new integrator.
\end{theorem}

\begin{theorem}\label{thm:new-robust}
The robust error bound of Theorem~$\ref{thm:proj-split-robust}$ holds verbatim also for the new integrator.
\end{theorem}

	As in \cite[Section 2.6.3]{KieriLubichWalach}, it can be further shown that an inexact solution of the matrix differential equations in the projector-splitting integrator
	leads to an additional error that is bounded in terms of the local errors in the inexact substeps, again with constants that do not depend on small singular values.

\subsection{Proof of Theorem~\ref{thm:new-exact}}

For the proof of Theorem \ref{thm:new-exact} we need the following auxiliary result, which extends an analogous result in \cite{CeL20} for symmetric matrices.

	\begin{lemma} \label{lem:Klem-exact}
		Let $\bfA(t) \in \mathbb{R}^{m \times n}$ be of rank~$r$  for $t_0 \leq t \leq t_1$, so that $\bfA(t)$ has a factorization \eqref{USV}, $\bfA(t)=\bfU(t)\bfS(t)\bfV(t)^\top$. Moreover, assume  that the $r\times r$ matrix $ \bfV(t_1)^\top \bfV(t_0)$ is invertible. Then,
		$$ \bfU_1 \bfU_1^\top \bfA(t_1) = \bfA(t_1).$$
	\end{lemma}

	\begin{proof}
		The solution  of the K-step at time $t_1$ is 
		\begin{equation*}
			\textbf{K}(t_1) = \bfA(t_0) \bfV_0 + \big( \bfA(t_1) - \bfA(t_0) \big) \bfV_0 = \bfA(t_1) \bfV_0 .
		\end{equation*}
		Hence,
		$$ \textbf{K}(t_1) = \bfU(t_1) \big[ \bfS(t_1)  (\bfV(t_1)^\top \bfV(t_0)) \big] \ . $$
		By assumption, the factor in big square brackets is invertible. Computing a QR-decomposition of this term, we have
		$$ \textbf{K}(t_1) = \bfU(t_1) \textbf{Q} \textbf{R} \ ,$$
		where $\textbf{Q} \in \R^{r \times r}$ is an orthogonal matrix and $\textbf{R} \in \R^{r \times r}$ is invertible and upper triangular.
		The QR-factorization of $\textbf{K}(t_1)$ thus yields
		$$\bfU_1 = \bfU(t_1) \textbf{Q} \in \R^{m \times r} .$$	
		To conclude,
		$$ \bfU_1 \bfU_1^\top \bfA(t_1) = \bfU(t_1) \textbf{Q} \textbf{Q}^\top \bfU(t_1)^\top \bfA(t_1) = \bfU(t_1) \bfU(t_1)^\top \bfA(t_1) = \bfA(t_1) ,$$
		which is the stated result.
	\qed \end{proof}

	\begin{proof} (of Theorem~\ref{thm:new-exact})
              Since the L-step is the K-step for the transposed matrix $\bfA(t)^\top$, which has the same rank as $\bfA(t)$,
		it follows from Lemma~\ref{lem:Klem-exact} that
		\begin{equation} \label{UUA}
		\bfU_1 \bfU_1^\top \bfA(t_1) = \bfA(t_1), \qquad
		\bfV_1 \bfV_1^\top \bfA(t_1)^\top = \bfA(t_1)^\top. 
                 \end{equation}		
             		The integrator yields in the S-step 
		$$ \bfS_1 = \bfU_1^\top \bfY_0 {\bfV}_1 + \bfU_1^\top (\bfA(t_1) - \bfA(t_0)) {\bfV}_1 =
		\bfU_1^\top \bfA(t_1)   {\bfV}_1, $$
		since $\bfY_0=\bfA(t_0)$.
		The result  after a time step  of the new integrator is
		$$ \bfY_1 =  \bfU_1 \bfS_1 \bfV_1^\top=\bfU_1 \bfU_1^\top \bfA(t_1)  (\bfV_1 \bfV_1^\top) = \bfA(t_1) ,$$
		where the last equality holds because of \eqref{UUA}. 
	\qed \end{proof}

\subsection{Proof of Theorem~\ref{thm:new-robust}}

Under the assumptions of Theorem~\ref{thm:proj-split-robust}, we introduce the quantity   
	\begin{equation}\label{vartheta}
	  \vartheta :=  (4e^{Lh} BL  + 9BL)h^2 + (3e^{Lh}+4)\varepsilon h +e^{Lh} \delta \, ,
	\end{equation}
	which is the local error bound of the projector-splitting integrator after one time step, as proved in  \cite[Theorem 2.1]{KieriLubichWalach}.

	\begin{lemma} \label{lem:Klem-errbound}
 		Let $\bfA_1$ be the solution at time $t_1=t_0+h$ of the full problem (\ref{eq:fullEq-mat}) with initial condition $\bfA_0$. Assume  that conditions $1$.-$\,3$. of Theorem~$\ref{thm:proj-split-robust}$ are fulfilled.  Then,
				$$ \| \bfU_1 \bfU_1^\top \bfA_1 - \bfA_1 \| \leq \vartheta.$$
	\end{lemma}

	\begin{proof} The result is proved in the course of the proof of Lemma~1 in \cite{CeL20}. We give the proof here for the convenience of the reader.
		The local error analysis in \cite{KieriLubichWalach} shows that the $r\times n$ matrix 
		$\bfZ=\bfS_1^\mathrm{ps}\bfV_1^{\mathrm{ps},\top}$, where $\bfS_1^\mathrm{ps}$ and $\bfV_1^\mathrm{ps}$ are the matrices computed in the third substep of the projector-splitting algorithm, satisfies
		\begin{equation*}
		\| \bfU_1 \textbf{Z} - \bfA_1 \| \leq \vartheta .
		\end{equation*}
		The square of the left-hand side can be split into two terms:
		\begin{equation*}
		\begin{aligned}
		\| \bfU_1 \textbf{Z} - \bfA_1 \|^2  
		&= \| \bfU_1 \textbf{Z} - \bfU_1 \bfU_1^\top \bfA_1 + \bfU_1 \bfU_1^\top \bfA_1 - \bfA_1 \| ^2 \\
		&= \| \bfU_1 \bfU_1^\top (\bfU_1 \textbf{Z} - \bfA_1) + (\textbf{I} - \bfU_1 \bfU_1^\top) \bfA_1 \| ^2 \\
		&= \| \bfU_1 \bfU_1^\top (\bfU_1 \textbf{Z} - \bfA_1) \| ^2 +  \|(\textbf{I} - \bfU_1 \bfU_1^\top) \bfA_1 \| ^2 .\\
		\end{aligned}
		\end{equation*} 
		Hence, 
		$$ \| \bfU_1 \bfU_1^\top(\bfU_1 \textbf{Z} - \bfA_1) \| ^2 +  \|(\textbf{I} - \bfU_1 \bfU_1^\top) \bfA_1 \| ^2 \leq \vartheta^2 . $$
		This yields the stated result for the second term.
	\qed \end{proof}


	\begin{lemma}\label{lem:gen-aux}
		Let $\bfA_1$, $\bfU_1$ and $\bfV_1$ be defined as above. The following estimate holds:
		$$ \| \bfU_1 \bfU_1^\top \bfA_1 \bfV_1 \bfV_1^\top - \bfA_1 \| \leq 2\vartheta. $$
	\end{lemma}

	\begin{proof} The L-step is the K-step for the transposed function $\bfG(t, \bfY)=\bfF(t,\bfY^\top)^\top$, which again
		fulfills conditions $1$.-$\,3$. of Theorem~$\ref{thm:proj-split-robust}$. 
		Conditions $1$. and $3$. hold because of the invariance of the Frobenius norm under transposition. 
		Condition~$2$. holds because
		\begin{equation*}
				\| (\bfI - \P(\bfY)) \bfG(t, \bfY) \|  
				= \| (\bfI - \P(\bfY^\top)) \bfF(t, \bfY^\top) \| 
				\leq \varepsilon,
		\end{equation*}
		where we used the identity $ \P(\bfY)\bfZ^\top = \big[ \P(\bfY^\top)\bfZ \big]^\top$. 
		From Lemma~\ref{lem:Klem-errbound} we thus have
		\begin{equation}\label{eq:U1V1-bound}
		\begin{aligned}
		&\| \bfU_1 \bfU_1^\top \bfA_1 - \bfA_1 \| \leq \vartheta,
		\\
		&\| \bfV_1 \bfV_1^\top \bfA_1^\top - \bfA_1^\top \| \leq \vartheta.
		\end{aligned}
		\end{equation}
		This implies that
		\begin{align*}
		\| \bfU_1 \bfU_1^\top \bfA_1 \bfV_1 \bfV_1^\top- \bfA_1 \| 
		&\leq 
		\| \bfU_1 \bfU_1^\top \bfA_1 \bfV_1 \bfV_1^\top -\bfA_1 \bfV_1 \bfV_1^\top + \bfA_1 \bfV_1 \bfV_1^\top - \bfA_1 \| 
		\\ &\leq
		\| \bfU_1 \bfU_1^\top \bfA_1 \bfV_1 \bfV_1^\top -\bfA_1 \bfV_1 \bfV_1^\top \| + \| \bfA_1 \bfV_1 \bfV_1^\top - \bfA_1 \| 
		\\ &\leq
		\| \big( \bfU_1 \bfU_1^\top \bfA_1 -\bfA_1 \big) \bfV_1 \bfV_1^\top \| + \| \bfV_1 \bfV_1^\top \bfA_1^\top  - \bfA_1^\top \|
		\\ &\leq
		\|  \bfU_1 \bfU_1^\top \bfA_1 -\bfA_1 \| \cdot \| \bfV_1 \bfV_1^\top \|_2 + \| \bfV_1 \bfV_1^\top \bfA_1^\top  - \bfA_1^\top \|.
		\end{align*}
		Since $\| \bfV_1 \bfV_1^\top \|_2=1$, the result follows from (\ref{eq:U1V1-bound}).
	\qed \end{proof}
	
	In the following lemma, we show that the approximation given after one time step is $O(h(h + \varepsilon))$ close to the solution of system (\ref{eq:fullEq-mat}) when the starting values coincide.
	
	\begin{lemma}[Local Error] \label{lem:loc-err} If $\bfA_0=\bfY_0$, 
		the following local error bound holds:
		$$ \| \bfY_1 - \bfA_1 \| \leq h(\hat c_1 \varepsilon  + \hat c_2 h) , $$
		where the constants only depend on $L$ and $B$ and a bound of the step size. In particular, the constants are independent of singular values of the exact or approximate solution. 
	\end{lemma}
	
	\begin{proof}
		With a few crucial modifications, the proof is similar to that of \cite[Lemma 2]{CeL20}. We report here the full proof for completeness and convenience of the reader. 
				By the identity $\bfY_1=\bfU_1\bfS_1\bfV_1^\top$ and Lemma~\ref{lem:gen-aux} we have that
		\begin{equation*}
		\begin{aligned}
		\| \bfY_1 - \bfA_1 \| 
		&\leq \| \bfY_1 - \bfU_1 \bfU_1^\top \bfA_1 \bfV_1 \bfV_1^\top \| + \| \bfU_1 \bfU_1^\top \bfA_1 \bfV_1 \bfV_1^\top- \bfA_1 \| \\
		&\leq \| \bfU_1( \bfS_1 - \bfU_1^\top \bfA_1 {\bfV}_1) \bfV_1^\top \| + 2\vartheta \\
		&= \| \bfS_1 - \bfU_1^\top \bfA_1  {\bfV}_1 \|  + 2\vartheta.
		\end{aligned}
		\end{equation*}
		The analysis of the local error thus reduces to estimating $\| \bfS_1 - \bfU_1^\top \bfA_1  {\bfV}_1 \|$.
		To this end, we introduce the following quantity: for $t_0\le t \le t_1$,
		$$ \widetilde\bfS(t) := \bfU_1^\top \bfA(t)  {\bfV}_1 . $$
		We write
		\begin{equation*}
		\begin{aligned}
		\bfA(t) 
		&= \bfU_1 \bfU_1^\top \bfA(t) \bfV_1 \bfV_1^\top + \Bigl( \bfA(t) - \bfU_1 \bfU_1^\top \bfA(t) \bfV_1 \bfV_1^\top \Bigr) 
		= \bfU_1 \widetilde\bfS(t) \bfV_1^\top + \textbf{R}(t),
		\end{aligned}
		\end{equation*}
		where $\textbf{R}(t)$ denotes the term in big brackets.
		Lemma~\ref{lem:gen-aux} and  the bound $B$ of $\bfF$ yield,  for $t_0 \le t \le t_1$,
		$$
		\| \bfA(t) - \bfA(t_1) \| \le \int_{t_0}^{t_1} \| \dt \bfA(s) \|\, ds =
		\int_{t_0}^{t_1} \| \bfF(s,\bfA(s)) \| \, ds \le Bh.
		$$
		Hence the remainder term is bounded by
		$$ 
		\| \textbf{R}(t) \| 
		\leq \|\textbf{R}(t) - \textbf{R}(t_1) \| + \| \textbf{R}(t_1) \|
		\leq 2Bh + 2\vartheta. 
		$$
		It follows that $\bfF(t, \bfA(t))$ can be written as
		\begin{equation*}
		\begin{aligned}
		\bfF(t, \bfA(t)) 
		&= \bfF(t, \bfU_1 \widetilde\bfS(t) \bfV_1^\top + \textbf{R}(t) ) \\
		&= \bfF(t, \bfU_1 \widetilde\bfS(t) \bfV_1^\top) + \bfD(t)
		\end{aligned}
		\end{equation*}
		with the defect 
		$$
		\bfD(t) := \bfF(t, \bfU_1 \widetilde\bfS(t) \bfV_1^\top + \textbf{R}(t)) - \bfF(t, \bfU_1 \widetilde\bfS(t) \bfV_1^\top).
		$$
		With the Lipschitz constant $L$ of $\bfF$, the defect is bounded by
		$$
		\|  \bfD(t) \| \le L \| \textbf{R}(t) \| \le 2L (Bh + \vartheta).
		$$
		We compare the two differential equations
		\begin{equation*}
		\begin{aligned}
		&\dot{\widetilde\bfS}(t) = \bfU_1^\top \bfF(t, \bfU_1 \widetilde\bfS(t) \bfV_1^\top)  {\bfV}_1 + \bfU_1^\top \bfD(t)  {\bfV}_1, 
		\qquad
		&\widetilde\bfS(t_0) = \bfU_1^\top \bfY_0  {\bfV}_1,\\
		&\dot{\bfS}(t) = \bfU_1^\top \bfF(t, \bfU_1 \bfS(t) \bfV_1^\top)  {\bfV}_1, 
		\qquad
		&\bfS(t_0) = \bfU_1^\top \bfY_0  {\bfV}_1.
		\end{aligned}
		\end{equation*}
		By construction, the solution of the first differential equation at time $t_1$ is  $ \widetilde \bfS(t_1) = \bfU_1^\top \bfA_1  {\bfV}_1$. The solution of the second differential equation is $\bfS_1$  as given by the S-step of the integrator.
		With the Gronwall inequality we obtain
		$$
		\|  \bfS_1 - \bfU_1^\top \bfA_1  {\bfV}_1 \| 
		\leq \int_{t_0}^{t_1} e^{L(t_1-s)} \, \| \bfD(s) \| \, ds
		\leq e^{Lh} \,2L (Bh + \vartheta) h.
		$$
		The result now follows using the definition of $\vartheta$. 
	\qed \end{proof}
	
	Using the Lipschitz continuity of the function $\bfF$,  we pass from the local to the global errors by the standard argument of Lady Windermere's fan \cite[Section II.3]{HairerNorsettWanner:ODE_BOOK1} and thus conclude the proof of Theorem~\ref{thm:new-robust-ten}.
	\subsection{Symmetric and skew-symmetric low-rank matrices}
	We now assume that the right-hand side function in \eqref{eq:fullEq-mat} is such that one of the following conditions holds,
	
	\begin{equation} \label{F-sym}
		\bfF(t,\bfY^\top)^\top =  \bfF(t, \bfY) \qquad \text{for all }\ \bfY \in \R^{n \times n}
	\end{equation}
	or
	\begin{equation} \label{F-skewsym}
	\bfF(t,\bfY^\top)^\top =  -\bfF(t, -\bfY) \qquad \text{for all }\ \bfY \in \R^{n \times n}.
	\end{equation}
	Under these conditions, solutions to \eqref{eq:fullEq-mat} with symmetric or skew-symmetric initial data remain symmetric or skew-symmetric, respectively, for all times.
	We also have preservation of (skew-)symmetry for the new integrator, which does not hold for the projector-splitting integrator.
	
	\begin{theorem}
		Let $\bfY_0 = \bfU_0 \bfS_0 \bfU_0^\top \in \R^{n \times n}$ be  symmetric or skew-symmetric and assume that the function $\bfF$ satisfies property~$(\ref{F-sym})$ or $(\ref{F-skewsym})$, respectively.
		Then, the approximation $\bfY_1$ obtained after one time step of the new integrator is symmetric or skew-symmetric, respectively.
	\end{theorem}
	
	\begin{proof} Let us just consider the skew-symmetric case \eqref{F-skewsym}. (The symmetric case is analogous.)
The L-step is the K-step for the transposed function $\bfG(t, \bfY)=\bfF(t,\bfY^\top)^\top$, and so the skew-symmetry of $\bfS_0$ and property \eqref{F-skewsym} imply that $\bfL(t_1)=-\bfK(t_1)$, which further yields $\bfV_1=\bfU_1$ and $\bfM=\bfN$. These identities show that the initial value and the right-hand side function of the differential equation for $\bfS(t)$ are skew-symmetric, which implies that $\bfS(t)$ and hence $\bfS_1$ are still skew-symmetric. Altogether, the algorithm gives us the skew-symmetric result $\bfY_1=\bfU_1 \bfS_1 \bfU_1^\top$.
	\qed \end{proof}
	
	 Under condition~(\ref{F-sym}) or~(\ref{F-skewsym}), the new integrator coincides with the (skew)-symmetry preserving low-rank matrix integrator of~\cite{CeL20}.

	\section{Recap: the Tucker tensor projector-splitting integrator}
	
     The solution $A(t)\in \R^{n_1\times\dots\times n_d}$ of a tensor differential equation
	\begin{equation} \label{eq:fullEq-ten}
	\dt{A}(t) = F(t, A(t)), 
	\qquad
	A(0) = A_0 
	\end{equation}	
	is approximated by the solution $Y(t)\in \R^{n_1\times\dots\times n_d}$ of multilinear rank $\bfr=(r_1,\dots,r_d)$ of the differential equation projected to the tangent space of the manifold of rank-$\bfr$ tensors at the current approximation (\cite{KochLubich10}, cf.~also \cite{Beck-etal:MCTDH}),
	\begin{equation} \label{eq:projEq-ten}
	\dt{Y}(t) = \P(Y(t)) F(t, Y(t)),
	\qquad
	Y(t_0) = Y_0,
	\end{equation}
	where $Y_0$ is a rank-$\bfr$ approximation to $A_0$. 
	Tensors $Y(t)$ of multilinear rank $\bfr$ are represented in the Tucker form \cite{DeLauthawer:HOSVD}, written here in a notation following \cite{KoldaBader:TensorDec}:
	\begin{align} \label{Tucker}
	&Y(t) = C(t) \bigtimes_{i=1}^d \bfU_i(t) , 
	\\ 
	& i.e.,\quad y_{i_1,\dots,i_d}(t) = \sum_{j_1,\dots,j_d} c_{j_1,\dots,j_d}(t) \,u_{i_1,j_1}(t)\dots u_{i_d,j_d}(t),
	\nonumber
	\end{align}
	where the slim basis matrices $\bfU_i \in \mathbb{R}^{n_i \times r_i}$ have orthonormal columns and
	the smaller core tensor $C(t) \in \mathbb{R}^{r_1 \times \dots \times r_d }$ is of full multilinear rank $\bfr$.

	The orthogonal tangent space projection $\P(Y)$ is given as an alternating sum of $2d-1$ subprojections \cite{Lubich:MCTDH}, and like in the matrix case, a projector-splitting integrator with favourable properties can be formulated and efficiently implemented \cite{Lubich:MCTDH,LubichVandWalach}. 
	The algorithm runs through the modes $i=1,\dots,d$ and solves differential equations for matrices of the dimension of the slim basis matrices and for the core tensor, alternating with orthogonalizations of slim matrices. Like the matrix projector-splitting integrator,
	also the Tucker tensor projector-splitting integrator has the exactness property and a robust error bound independently of small singular values of matricizations of the core tensor  \cite[Theorems 4.1 and 5.1]{LubichVandWalach}.

	\section{A new robust low-rank Tucker tensor integrator}
	The low-rank numerical integrator defined in Section \ref{sec:Par} for the matrix case extends in a natural way to the Tucker tensor format, and this extension still has the exactness property and robust error bounds that are independent of small singular values of matricizations of the core tensor.
	
In comparison with the Tucker integrator of \cite{Lubich:MCTDH} and \cite{LubichVandWalach}, the new Tucker tensor integrator has the following favourable properties:
	\begin{enumerate}
	\item The solution of the differential equations for the $n_i\times r_i$ matrices can be done in parallel for $i=1,\dots,d$, and also the  QR decompositions can be done in parallel.
	\item No differential equations are solved backward in time. No differential equations for $r_i\times r_i$ matrices need to be solved.
	\item The integrator preserves (anti-)symmetry if the differential equation does.
	\end{enumerate}
On the other hand, in contrast to the projector-splitting Tucker integrator there is apparently no efficient way to construct a time-reversible integrator from this new Tucker integrator.

	\subsection{Formulation of the algorithm}
		One time step of integration from time $t_0$ to $t_1=t_0+h$  starting from a  Tucker tensor 
		of multilinear rank $(r_1,\dots,r_d)$ in factorized form, 
		$Y_0 = C_0 \bigtimes_{i=1}^d \bfU_i^0$,
	 computes an updated Tucker tensor 
		of multilinear rank $(r_1,\dots,r_d)$ in factorized form, $Y_1 = C_1 \bigtimes_{i=1}^d \bfU_i^1$, in the following way:
		
		\begin{enumerate}
		\item Update the basis matrices $\bfU_i^0 \to \bfU_i^1$ for $i=1,\dots,d$ in parallel:\\[2mm]
		Perform a QR factorization of the transposed $i$-mode matricization of the core tensor:
		$$ \text{\textbf{Mat}}_i(C_0)^\top = \textbf{Q}_i \bfS_i^{0,\top} .$$
		With 
		$ \bfV_i^{0,\top} = 
		\textbf{Q}_i^{\top} \bigotimes_{j \neq i}^d \bfU_j^{0,\top} \in \R^{r_i \times n_{\neg i}} 
		$
		(which yields 
		$ {\textbf{Mat}}_i(Y_0) = \bfU_i^0 \bfS_i^0 \bfV_i^{0, \top} $)\\[2mm]
		and the  matrix function $\bfF_{i}(t, \cdot) := \text{\textbf{Mat}}_i \circ F(t, \cdot) \circ  \textit{Ten}_i$,
		integrate from $t=t_0$ to $t_1$ the $n_i \times r_i$ matrix differential equation
		$$ 
		\dot{\bfK}_i(t) = \bfF_{i}(t,\bfK_i(t) \bfV_i^{0,\top}) \bfV_i^0,
		\qquad \bfK_i(t_0) = \bfU_i^0 \bfS_i^0.
		$$
		Perform a QR factorization $\textbf{K}_i(t_1) = \bfU_i^1 {\bfR}_i^1$ and compute the $r_i\times r_i$ matrix $\bfM_i= \bfU_i^{1,\top} \bfU_i^0$.

		\item Update the core tensor $C_0\to C_1$:\\[2mm]
		Integrate from $t=t_0$ to $t_1$ the $r_1 \times\dots\times r_d$  tensor differential equation
			\begin{align*}
			&\dot{C}(t) = F \left( t, C(t) \bigtimes_{i=1}^d \bfU_i^1 \right) \bigtimes_{i=1}^d \bfU_i^{1,\top},
			\quad C(t_0)  = C_0 \bigtimes_{i=1}^d \bfM_i 
			\end{align*}
			and set $C_1=C(t_1)$.		
		\end{enumerate}
	
	To continue in time,  we take $Y_1$ as starting value for the next step and perform another step of the integrator. 
	
	We observe that, in contrast to the Tucker integrators of \cite{LubichVandWalach,Lubich:MCTDH},  the factors $\bfU_i \in \mathbb{R}^{n_i \times r_i}$ are updated simultaneously for $i=1,\dots, d$.

	\subsection{Exactness property}
	The following result extends the exactness results of Theorem~\ref{thm:new-exact} and \cite[Theorem 4.1]{LubichVandWalach} to the new Tucker tensor integrator. 

	\begin{theorem}[Exactness property]
		\label{thm:sym-exact-ten}
		Let $A(t) = C(t) \bigtimes_{i=1}^d \bfU_i(t)$ be of multilinear rank $(r_1,\dots,r_d)$ for $t_0 \leq t \leq t_1$. Moreover, assume  that the $r_i\times r_i$ matrix $ \bfU_i(t_1)^\top \bfU_i(t_0)$ is invertible for each $i=1, \dots, d$. With $Y_0=A(t_0)$, the new Tucker integrator  with rank $(r_1,\dots,r_d)$ for
		$\dt Y(t)=\P(Y(t))\dt A(t)$ with starting value $Y_0=A(t_0)$ is then exact: $Y_1 = A(t_1)$.
	\end{theorem}
	
	\begin{proof}
		For each $i=1, \dots, d$, we apply Lemma~\ref{lem:Klem-exact} to $\textbf{Mat}_i(\dt A(t))$
		$$ 
			\textbf{Mat}_i(A(t_1) \times_i \bfU_i^1 \bfU_i^{1,\top}) 
			= \bfU_i^1 \bfU_i^{1,\top} \textbf{Mat}_i(A(t_1)) 
			= \textbf{Mat}_i(A(t_1)) .
		$$  
		We tensorize  in the $i$-th mode and obtain
		$$ A(t_1) \times_i \bfU_i^1 \bfU_i^{1,\top} = A(t_1), \qquad  i = 1, \dots , d \ .$$
		With $Y_0 = A(t_0)$ 
		we obtain from the second substep of the algorithm
		\begin{align*} 
			Y_1 &= C_1 \bigtimes_{i=1}^d \bfU_i^1  
			\\&=
		\Bigl( Y_0 \bigtimes_{i=1}^d  \bfU_i^{1,\top} + (A(t_1)-A(t_0)) \bigtimes_{i=1}^d  \bfU_i^{1,\top} \Bigr) \bigtimes_{i=1}^d \bfU_i^1
		\\
		&= \Bigl( A(t_1)\bigtimes_{i=1}^d  \bfU_i^{1,\top} \Bigr) \bigtimes_{i=1}^d \bfU_i^1 
		\\&= A(t_1) \bigtimes_{i=1}^d \bfU_i^1 \bfU_i^{1,\top} = A(t_1)\, ,
		\end{align*}		
		which proves the exactness. 
	\qed \end{proof}
	
	\subsection{Robust error bound} 
	The robust error bounds from Theorem~\ref{thm:new-robust} and \cite[Theorem~5.1]{LubichVandWalach} extend to the new Tucker tensor integrator as follows. 
	The norm $\|B\|$ of a tensor $B$ used here is the Euclidean norm of the vector of entries of $B$.

	\begin{theorem}[{Robust error bound}]
		\label{thm:new-robust-ten}
		Let $A(t)$ denote the solution of the tensor differential equation \eqref{eq:fullEq-ten}. Assume  the following:
		\begin{enumerate}
			\item 
			$F$ is Lipschitz-continuous and bounded.		
			\item
			The non-tangential part of $F(t, Y)$ is $\varepsilon$-small:
			$$
			\| (I - \P(Y)) F(t, Y) \| \le \eps
			$$
			for all $Y$ of multilinear rank  $(r_1,\dots,r_d)$ in a neighbourhood of $A(t)$ and $0\le t \le T$.
			\item
			The error in the initial value is $\delta$-small:
			$$
			\| Y_0 - A_0 \| \le \delta.
			$$
		\end{enumerate}	
		Let $Y_n$ denote the approximation of multinear rank $(r_1,\dots,r_d)$ to $A(t_n)$ at $t_n=nh$ obtained after n steps of the new Tucker integrator with step-size $h>0$.
		Then, the error satisfies for all $n$ with $t_n =  nh \leq T$
		$$ \| Y_n - A(t_n) \| \leq c_0\delta + c_1 \varepsilon + c_2 h ,$$	
		where the constants $c_i$ only depend on the Lipschitz constant $L$ and bound $B$ of $F$, on~$T$, and on the dimension $d$. In particular, the constants are independent of singular values of matricizations of the exact or approximate solution. 
	\end{theorem}
	
	
	The proof of Theorem \ref{thm:new-robust-ten} proceeds similar to the proof of Theorem~\ref{thm:new-robust} for the matrix case. We begin with two key lemmas and are then in a position to analyse  the local error produced after one time step. We denote the solution value at $t_1$ by $A_1$. The basis matrix computed in the first part of the integrator is denoted by $\bfU_i^1$ for each $i=1,\dots, d$. 	
	\begin{lemma} \label{lem:Tucker:conditionsFi}
		For each $i=1,\dots, d$, the function $\bfF_{i}(t, \cdot) := {\mathbf{Mat}}_i \circ F(t, \cdot) \circ  \mathit{Ten}_i$ fulfills Conditions $1.$ - $2.$ of Theorem~$\ref{thm:proj-split-robust}$, and the initial matrix $\bfY_{(i)}^0=\mathbf{Mat}_i(Y_0)$ fulfills Condition~$3$. of that theorem.
	\end{lemma}

	\begin{proof}
		For each $i=1 \dots d$, it holds that for $\bfY_{(i)}=\mathbf{Mat}_i(Y)$,
		$$ 
			\| \bfF_{i}(t, \bfY_{(i)}) \| = \| F(t, Y) \| .
		$$
		The boundedness and Lipschitz condition of the matrix-valued function $\bfF_{i}$ follows from the boundedness and Lipschitz condition of the tensor-valued function~$F$.
		
		Condition~$2.$ follows with the help of the correspondingly defined projection
		$$
		\P_{i}(\bfY_{(i)}) := \text{\textbf{Mat}}_i \circ \P(Y) \circ  \textit{Ten}_i \quad\ \text{for}\quad \bfY_{(i)}=\mathbf{Mat}_i(Y),
		$$ 
		which is an orthogonal projection onto a subspace of the tangent space at $\bfY_{(i)}$ of the manifold of rank-$r_i$ matrices of dimension $n_i\times n_{\neg i}$. Denoting the orthogonal projection onto this tangent space by $\P_{(i)}(\bfY_{(i)})$, we thus have
		\begin{equation*}
			\| \big(\bfI - \P_{(i)}(\bfY_{(i)}) \big) \bfF_{i}(t, \bfY_{(i)}) \| \le \| \big(\bfI - \P_{i}(\bfY_{(i)}) \big) \bfF_{i}(t, \bfY_{(i)}) \|
			= \| (\bfI - \P(Y )) F(t,Y) \| \leq \varepsilon .
		\end{equation*}
		Condition~$3.$ holds due to the invariance of the Frobenius norm under matricization,
		$$ \| \bfY_{(i)}^0 - \text{\textbf{Mat}}_i(A_0) \| =  \|  \text{\textbf{Mat}}_i(Y_0-A_0)\| =  \| Y_0 - A_0 \| \leq \delta , $$
		and so we obtain the stated result.
	\qed \end{proof}
	
	\begin{lemma} \label{lem:aux-ten}
		The following estimate holds with $\vartheta$ of \eqref{vartheta}:
		$$ \| A_1 \bigtimes_{i=1}^d \bfU_i^1 \bfU_i^{1,\top} - A_1 \| \leq  d\,\vartheta, $$
		where $c$ only depends on $d$ and a bound for $hL$.
	\end{lemma}
	
	\begin{proof}
		
		From Lemma~\ref{lem:Tucker:conditionsFi} and Lemma~\ref{lem:Klem-errbound},
		$$ \| \bfU_i^1 \bfU_i^{1,\top} \textbf{Mat}_i(A_1) - \textbf{Mat}_i(A_1) \| \leq \vartheta ,\qquad i=1,\dots, d  . $$
		The norm is invariant under tensorization and so the bound is equivalent to
		$$ \|  A_1 \times_i \bfU_i^1 \bfU_i^{1,\top} - A_1  \| \leq \vartheta ,\qquad i=1,\dots, d  . $$
		To conclude, we observe 
		\begin{equation*}
		\begin{aligned}
		&\|  A_1 \bigtimes_{i=1}^d \bfU_i^1 \bfU_i^{1,\top} - A_1 \| 
			\\
		&\leq \|  A_1 \bigtimes_{i=1}^d \bfU_i^1 \bfU_i^{1,\top}-A_1 \bigtimes_{i=1}^{d-1} \bfU_i^1 \bfU_i^{1,\top}
		+A_1 \bigtimes_{i=1}^{d-1} \bfU_i^1 \bfU_i^{1,\top}- A_1 \| \\
		&\leq \|  A_1 \bigtimes_{i=1}^d \bfU_i^1 \bfU_i^{1,\top}-A_1 \bigtimes_{i=1}^{d-1} \bfU_i^1 \bfU_i^{1,\top}\|
		+ \| A_1 \bigtimes_{i=1}^{d-1} \bfU_i^1 \bfU_i^{1,\top}- A_1 \| \\
		&\leq \| (A_1 \times_d \bfU_i^1 \bfU_i^{1,\top}- A_1) \bigtimes_{i=1}^{d-1} \bfU_i^1 \bfU_i^{1,\top}\| 
		+ \| A_1 \bigtimes_{i=1}^{d-1} \bfU_i^1 \bfU_i^{1,\top}- A_1 \| \\
		&\leq \| A_1 \times_d \bfU_i^1 \bfU_i^{1,\top}- A_1 \| 
		+ \| A_1 \bigtimes_{i=1}^{d-1} \bfU_i^1 \bfU_i^{1,\top}- A_1 \| \\
		&\leq \vartheta + \| A_1 \bigtimes_{i=1}^{d-1} \bfU_i^1 \bfU_i^{1,\top}- A_1 \| ,
		\end{aligned}
		\end{equation*}
		and the result follows by an iteration of this argument. 
	\qed \end{proof}
	
	We are now in a  position to analyse the local error produced after one time step of the integrator. 
	
	\begin{lemma}[Local error] If $A_0 = Y_0$, the following local error bound holds for the new Tucker tensor integrator:
		$$ \| Y_1 - A_1 \| \leq \hat c \,h(BLh+\eps) , $$
		where $\hat c$ only depends on $d$ and a bound of $hL$. In particular, the constant is independent of singular values of the exact or approximate solution. 
	\end{lemma}
	
	We omit the proof because, up to minor modifications analogous to those in the proof of Lemma~\ref{lem:loc-err}, the result follows as in \cite[Section 5.3]{CeL20} on using the two previous lemmas.

	
	Using the Lipschitz continuity of the function $F$, we pass from the local to the global errors by the standard argument of Lady Windermere's fan \cite[Section II.3]{HairerNorsettWanner:ODE_BOOK1} and thus conclude the proof of Theorem~\ref{thm:new-robust-ten}.
	
	\subsection{Symmetric and anti-symmetric low-rank Tucker tensors}
 For  permutations $\sigma\in S_d$, we use the notation $\sigma(Y)=\bigl(y_{i_{\sigma(1)},\dots,i_{\sigma(d)}} \bigr)$	for tensors $Y = (y_{i_1,\dots,i_d})\in \R^{n \times\dots\times n}$ of order $d$. A tensor $Y$ is called {\it symmetric} if $\sigma(Y)=Y$ for all $\sigma\in S_d$, and is called {\it anti-symmetric} if
$\sigma(Y)= (-1)^{\mathrm{sign}(\sigma)} \,Y$ for all $\sigma\in S_d$.
 
 	We now assume that the right-hand side function in \eqref{eq:fullEq-ten} is such that one of the following conditions holds: For all permutations $\sigma\in S_d$ and all tensors $Y \in \R^{n \times\dots\times n}$ of order $d$,
	\begin{equation} \label{F-sym-ten}
		\sigma\bigl( F(t,\sigma(Y))\bigr) =  F(t, Y) 
	\end{equation}
	or
	\begin{equation} \label{F-antisym-ten}
		\sigma\bigl( F(t,\sigma(Y))\bigr) =  (-1)^{\mathrm{sign}(\sigma)} \,F(t, Y) 
	\end{equation}
	Under these conditions, solutions to \eqref{eq:fullEq-mat} with symmetric or anti-symmetric initial data remain symmetric or anti-symmetric, respectively, for all times.
	We also have preservation of (anti-)symmetry for the new integrator, which does not hold for the projector-splitting integrator.
	
	\begin{theorem}
		Let $Y_0 $ be  symmetric or anti-symmetric and assume that the function $F$ satisfies property~$(\ref{F-sym-ten})$ or $(\ref{F-antisym-ten})$, respectively.
		Then, the approximation $Y_1$ obtained after one time step of the new integrator is symmetric or anti-symmetric, respectively.
	\end{theorem}
	
	The simple proof is similar to the matrix case and is therefore omitted.
	
		 Under condition~(\ref{F-sym-ten}) or~(\ref{F-antisym-ten}), the new integrator coincides with the (anti)-symmetry preserving low-rank Tucker tensor integrator of~\cite{CeL20}.

	\section{Numerical Experiments}
	In this section, we present results of different numerical experiments. 
	The experiments were done using Matlab R2017a software with TensorLab package v3.0~\cite{vervliet2016tensorlab}. 
	
	\subsection{Robustness with respect to small singular values}
	In the first example, the time-dependent matrix is given explicitly as
	$$ \textbf{A}(t) = \big( e^{t\textbf{W}_1} \big) e^t \textbf{D} \big( e^{t\textbf{W}_2} \big)^\top, \quad 0 \leq t \leq 1 \ .$$
	The matrix $\textbf{D} \in \R^{N \times N}$ is diagonal with entries $d_j = 2^{-j}$. The matrices $\textbf{W}_1 \in \R^{N \times N}$ and $\textbf{W}_2 \in \R^{N \times N}$ are skew-symmetric and randomly generated. We note that $e^t 2^{-j}$ are the singular values of $A(t)$. We choose $N=100$ and final time $T=1$.  We compare the new low-rank integrator presented in Section~\ref{sec:Par} with a numerical solution obtained with the classical fourth-order explicit Runge-Kutta method applied to the system of differential equations for dynamical low-rank approximation as derived in \cite{KochLubich07}. 
	
	The numerical results for different ranks are shown in Figure \ref{fig:explicitmatrix}. In contrast to the Runge--Kutta method, the new low-rank integrator does not require a step-size restriction in the presence of small singular values. The same favourable behaviour was shown for the projector-splitting integrator in \cite{KieriLubichWalach}.
	
	\begin{figure}
		\includegraphics[height=6cm, width=\textwidth]{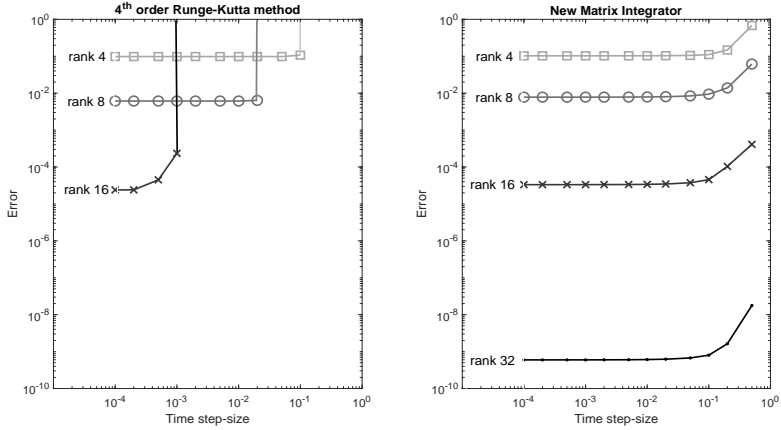}
		\caption{Comparison of the explicit Runge Kutta method (left) and the proposed new integrator (right) for	different approximation ranks and step sizes in the case of a given time-dependent matrix. }
		\label{fig:explicitmatrix}
	\end{figure}

	\subsection{Error behaviour}
	
	In the second example, we integrate a (non-stiff) discrete Schr\"odinger equation in imaginary time,
	$$ \dot{Y} = -\text{H}[Y], \quad Y(t_0) = C_0 \bigtimes_{i=1}^d \textbf{U}_i^0 \ .$$
	Here, 
	\begin{align*}
		& \text{H}[Y] = -\frac{1}{2}\sum_{j=1}^{d} \big( Y \times_j \textbf{D}\big) + Y \bigtimes_{i=1}^d \textbf{V}_{cos}  \in \R^{N \times \dots \times N} ,
		\\
		& \textbf{D} = \texttt{tridiag}(-1,2,-1) \in \R^{N \times N} , 
		\\
		& \textbf{V}_\text{cos} := \text{diag} \{ 1- \cos( \frac{2 \pi j }{N} ) \}, \quad j=-N/2, \dots, N/2-1 \ . 
	\end{align*}
	The function H arises from the Hamiltonian $\mathcal{H} = -\frac{1}{2}\Delta_{\mathrm{discrete}} + V(x)$ on a equidistant space grid with the torsional potential $V(x_1, \dots, x_d) = \prod_{i=1}^d (1-\text{cos}(x_i))$.
	
	For each $i=1,\dots,d$, the orthonormal matrices $\textbf{U}_i^0 \in \R^{N \times N}$ are randomly generated. The core tensor $C_0 \in \R^{N \times N \times N}$ has only non-zero diagonal elements set equal to $ (C_0)_{jjj} = 10^{-j} $ for $j=1, \dots N$ in the case $d=3$, and analogously in the matrix case $d=2$.
	
	The reference solution was computed with the Matlab solver \texttt{ode45} and stringent tolerance parameters \textsc{\{'RelTol', 1e-10, 'AbsTol', 1e-10\} }. The differential equations appearing in the definition of a step of the new matrix and Tucker integrators have all been solved either with a single step of a second- or fourth-order explicit Runge--Kutta method. 
	
		\begin{figure}[t]
		\includegraphics[height=5cm, width=\textwidth]{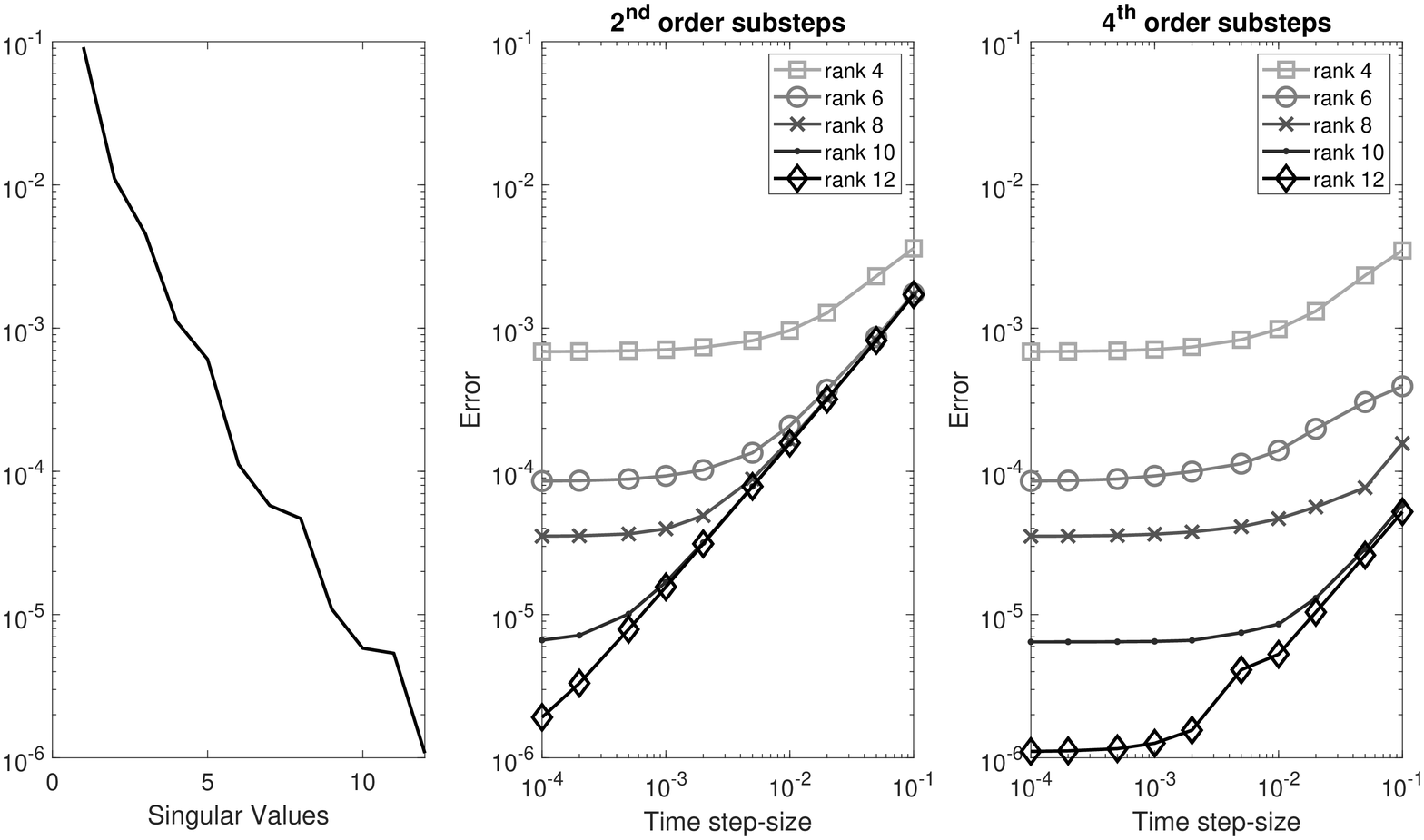}
		\caption{First twelve singular values of the reference solution at time $T=0.1$ and approximation errors for different ranks, time-integration methods in the substeps of the new matrix integrator, and step-sizes for the matrix differential equation ($d=2$).}
		\label{fig:MatDiffEq}
	\end{figure}

	\begin{figure}[h]
		\includegraphics[height=5cm, width=\textwidth]{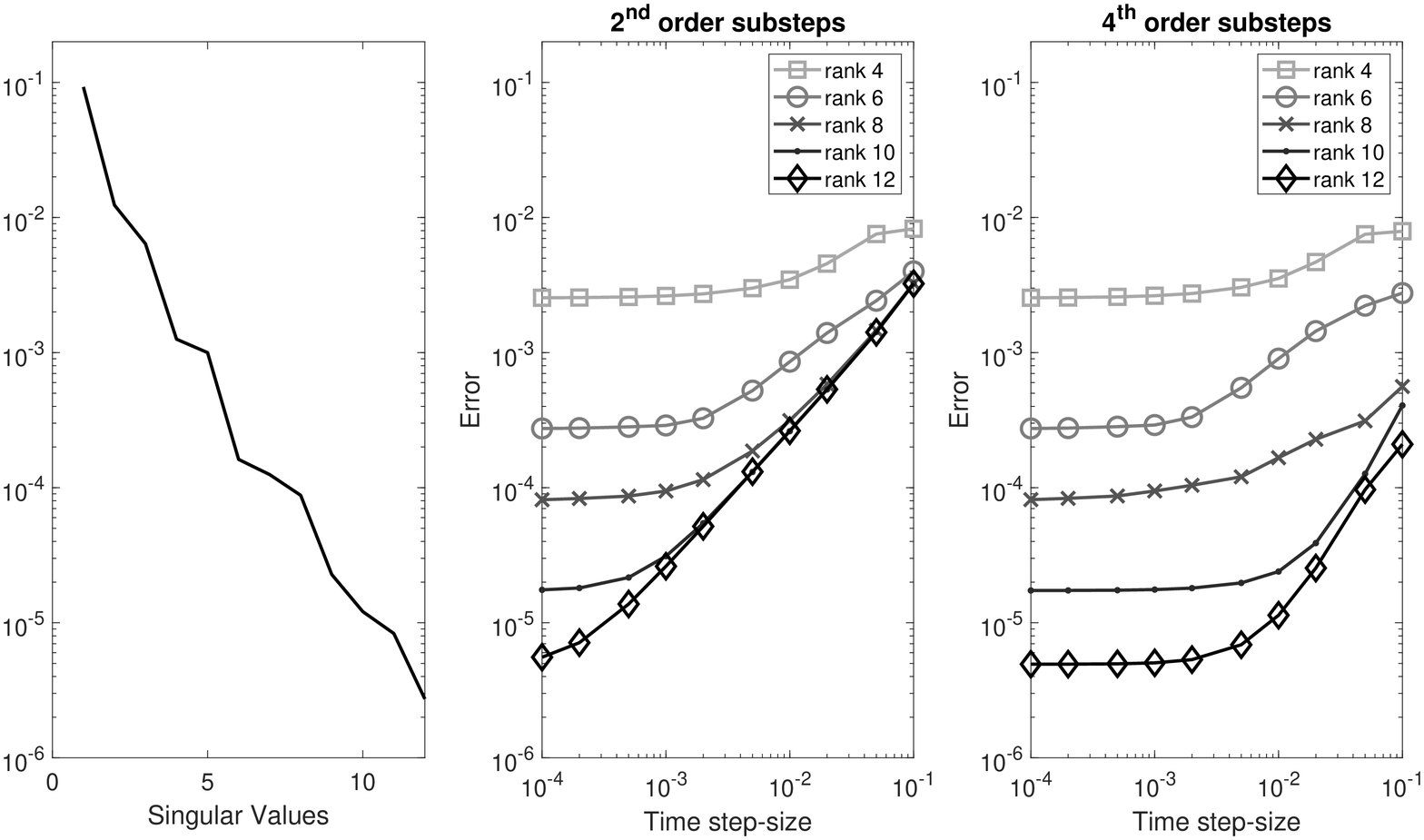}
		\caption{First twelve singular values of the matricisization in first mode of the reference solution at time $T=0.1$ and approximation errors for different multi-linear ranks, time-integration methods in the substeps of the new fixed-rank Tucker tensor integrator and step-sizes for the tensor differential equation ($d=3$).}
		\label{fig:TenDiffEq}
	\end{figure} 
	 
	We choose $N=100$, final time $T=0.1$ and $d=2,3$. The multi-linear rank is chosen such that $r_1 = r_2 = \dots =  r_d$.
	The singular values of the matricization in the first mode of the reference solution and the absolute errors $\| Y_n -A(t_n) \|_F$ at time $t_n=T$ of the approximate solutions for different ranks, calculated with different step-sizes and different time integration methods, are shown in Figure~\ref{fig:MatDiffEq} for the matrix case($d=2$), and in  Figure~\ref{fig:TenDiffEq} for the tensor case($d=3$). The figures clearly show that solving the substeps with higher accuracy allows us to take larger step-sizes to achieve a prescribed error.

	\subsection{Comparison with the matrix projector-splitting integrator over different ranks} 
	In the last example, we compare the matrix projector splitting integrator with the new matrix integrator of Section~\ref{sec:Par}. Here, the complex case is considered: in the definition of the sub-problems appearing in the new matrix integrator, it is sufficient to replace the transpose with the conjugate transpose.
	
	We consider a Schr\"odinger equation as in~\cite[Section 4.3]{KieriLubichWalach},
	\begin{align*}
		&i \partial_t u(x,t) = -\frac{1}{2} \Delta u(x,t) +\frac{1}{2} x^\top \! \textbf{A} x\, u(x,t), \quad x \in \mathbb{R}^2, t>0,
		\\
		& u(x,0) = \pi^{-\frac{1}{2}} \text{exp} \big(\frac{1}{2} x_1^2 + \frac{1}{2}(x_2-1)^2 \big),
		\\
		& \textbf{A} = \begin{pmatrix} 2 &-1\\ -1 & 3 \end{pmatrix} .
	\end{align*}
	The problem is discretized with a Fourier collocation method with a grid of $N \times N$ points; the solution is essentially supported within $\Omega = [-7.5, 7.5]^2	$. We choose the final time $T=5$ and $N=128$, which makes the problem moderately stiff.
	First, we compute  a reference solution with an Arnoldi method and a tiny time-step size $h=10^{-4}$.
	Then, for each rank from $1$ until $20$, we compute a low-rank approximation with the matrix projector splitting integrator and the new matrix integrator. 
	The lower-dimensional sub-problems appearing in the definition of the two integrators are solved with an Arnoldi method and time-step size $h=0.005$. For each rank, the absolute error in Frobenius norm of the two given approximations at the final time $T=5$, with respect to the reference solution, are shown in Figure~\ref{fig:NewvsKSL}.
	
	\begin{figure}
		\centering
		\includegraphics[width=\textwidth]{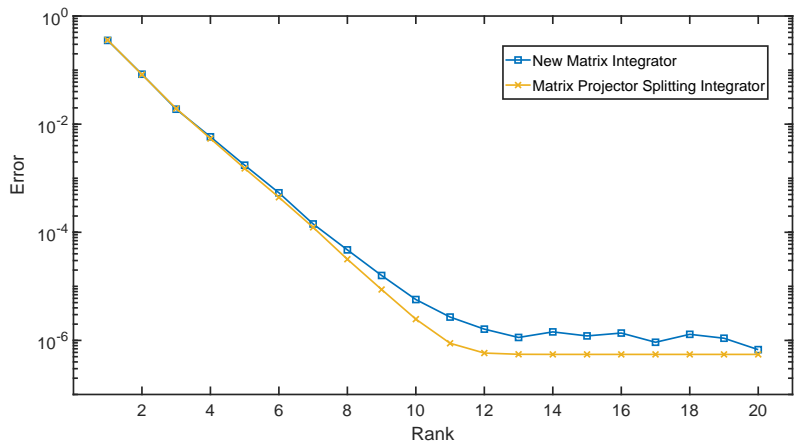}
		\caption{Approximation errors for different ranks at final time $T=5$ of the low-rank approximation computed with the matrix projector splitting integrator and the new matrix integrator.}
		\label{fig:NewvsKSL}
	\end{figure}
	
	\begin{acknowledgements}
		The last numerical example is based upon the original source code of \cite[Section 4.3]{KieriLubichWalach}; we would like to thank Hanna Walach for kindly providing it. 
		
		This work was funded by the Deutsche Forschungsgemeinschaft (DFG, German Research Foundation) --- Project-ID 258734477 --- SFB 1173. 
	\end{acknowledgements}
	
	\bibliographystyle{abbrv}
	\bibliography{dlrpar} 
	
\end{document}